\documentclass[11pt]{amsart}
\usepackage{amsopn}
\usepackage{amssymb, amscd}
\usepackage{multirow}

\newcommand{\nc}{\newcommand}

\nc{\fg}{\mathfrak{f} } \nc{\vg}{\mathfrak{v} } \nc{\wg}{\mathfrak{w} }
\nc{\zg}{\mathfrak{z} } \nc{\ngo}{\mathfrak{n} } \nc{\kg}{\mathfrak{k} }
\nc{\mg}{\mathfrak{m} } \nc{\bg}{\mathfrak{b} } \nc{\ggo}{\mathfrak{g} }
\nc{\ggob}{\overline{\mathfrak{g}} } \nc{\sog}{\mathfrak{so} }
\nc{\sug}{\mathfrak{su} } \nc{\spg}{\mathfrak{sp} } \nc{\slg}{\mathfrak{sl} }
\nc{\glg}{\mathfrak{gl} } \nc{\cg}{\mathfrak{c} } \nc{\rg}{\mathfrak{r} }
\nc{\hg}{\mathfrak{h} } \nc{\tg}{\mathfrak{t} } \nc{\ug}{\mathfrak{u} }
\nc{\dg}{\mathfrak{d} } \nc{\ag}{\mathfrak{a} } \nc{\pg}{\mathfrak{p} }
\nc{\sg}{\mathfrak{s} } \nc{\affg}{\mathfrak{aff} } \nc{\qg}{\mathfrak{q} }

\nc{\pca}{\mathcal{P}} \nc{\nca}{\mathcal{N}} \nc{\lca}{\mathcal{L}}
\nc{\oca}{\mathcal{O}} \nc{\mca}{\mathcal{M}} \nc{\tca}{\mathcal{T}}
\nc{\aca}{\mathcal{A}} \nc{\cca}{\mathcal{C}} \nc{\gca}{\mathcal{G}}
\nc{\sca}{\mathcal{S}} \nc{\hca}{\mathcal{H}} \nc{\bca}{\mathcal{B}}
\nc{\dca}{\mathcal{D}} \nc{\val}{\operatorname{val}}

\nc{\vp}{\varphi} \nc{\ddt}{\frac{d}{dt}} \nc{\dds}{\frac{d}{ds}}
\nc{\dpar}{\frac{\partial}{\partial t}} \nc{\im}{\mathrm{i}}

\nc{\SO}{\mathrm{SO}} \nc{\Spe}{\mathrm{Sp}} \nc{\Sl}{\mathrm{SL}}
\nc{\SU}{\mathrm{SU}} \nc{\Or}{\mathrm{O}} \nc{\U}{\mathrm{U}} \nc{\Gl}{\mathrm{GL}}
\nc{\Se}{\mathrm{S}} \nc{\Cl}{\mathrm{Cl}} \nc{\Spein}{\mathrm{Spin}}
\nc{\Pin}{\mathrm{Pin}} \nc{\G}{\mathrm{GL}_n(\RR)} \nc{\g}{\mathfrak{gl}_n(\RR)}

\nc{\RR}{{\Bbb R}} \nc{\HH}{{\Bbb H}} \nc{\CC}{{\Bbb C}} \nc{\ZZ}{{\Bbb Z}}
\nc{\FF}{{\Bbb F}} \nc{\NN}{{\Bbb N}} \nc{\QQ}{{\Bbb Q}} \nc{\PP}{{\Bbb P}} \nc{\OO}{{\Bbb O}}

\nc{\vs}{\vspace{.2cm}} \nc{\vsp}{\vspace{1cm}} \nc{\ip}{\langle\cdot,\cdot\rangle}
\nc{\ipp}{(\cdot,\cdot)} \nc{\la}{\langle} \nc{\ra}{\rangle} \nc{\unm}{\frac{1}{2}}
\nc{\unc}{\frac{1}{4}} \nc{\und}{\frac{1}{16}} \nc{\no}{\vs\noindent}
\nc{\lam}{\Lambda^2(\RR^n)^*\otimes\RR^n} \nc{\tangz}{{\rm T}^{\rm Zar}}
\nc{\nor}{{\sf n}}  \nc{\mum}{/\!\!/} \nc{\kir}{/\!\!/\!\!/}
\nc{\Ri}{\tfrac{4\Ric_{\mu}}{||\mu||^2}} \nc{\ds}{\displaystyle}
\nc{\ben}{\begin{enumerate}} \nc{\een}{\end{enumerate}} \nc{\f}{\frac}
\nc{\lb}{[\cdot,\cdot]} \nc{\isn}{\tfrac{1}{||v||^2}}
\nc{\gkp}{(\ggo=\kg\oplus\pg,\ip)} \nc{\ukh}{(\ug=\kg\oplus\hg,\ip)}
\nc{\tgkp}{(\tilde{\ggo}=\kg\oplus\pg,\ip)}
\nc{\wt}{\widetilde} \nc{\mm}{M}
\nc{\iop}{\mathtt{i}} \nc{\jop}{\mathtt{j}}

\nc{\Hess}{\operatorname{Hess}} \nc{\ad}{\operatorname{ad}}
\nc{\Ad}{\operatorname{Ad}} \nc{\rank}{\operatorname{rank}}
\nc{\Irr}{\operatorname{Irr}} \nc{\End}{\operatorname{End}}
\nc{\Aut}{\operatorname{Aut}} \nc{\Inn}{\operatorname{Inn}}
\nc{\Der}{\operatorname{Der}} \nc{\Ker}{\operatorname{Ker}}
\nc{\Iso}{\operatorname{Iso}} \nc{\Diff}{\operatorname{Diff}}
\nc{\Lie}{\operatorname{L}} \nc{\tr}{\operatorname{tr}} \nc{\dif}{\operatorname{d}}
\nc{\sen}{\operatorname{sen}} \nc{\modu}{\operatorname{mod}}
\nc{\CRic}{\operatorname{PP}} \nc{\Cric}{\operatorname{P}} \nc{\Ricci}{\operatorname{Ric}}
\nc{\sym}{\operatorname{sym}} \nc{\herm}{\operatorname{herm}} \nc{\symac}{\operatorname{sym^{ac}}}
\nc{\symc}{\operatorname{sym^{c}}} \nc{\scalar}{\operatorname{scal}}
\nc{\grad}{\operatorname{grad}} \nc{\ricci}{\operatorname{Rc}}
\nc{\Nor}{\operatorname{Norm}}  \nc{\ricc}{\operatorname{Rc^{c}}}
\nc{\Ricc}{\operatorname{Ric^{c}}} \nc{\ricac}{\operatorname{Rc^{ac}}}
\nc{\Ricac}{\operatorname{Ric^{ac}}} \nc{\Riem}{\operatorname{R}}
\nc{\riccig}{\operatorname{ric^{\gamma}}} \nc{\Rin}{\operatorname{M}}
\nc{\Le}{\operatorname{L}} \nc{\tang}{\operatorname{T}}
\nc{\level}{\operatorname{level}} \nc{\rad}{\operatorname{r}}
\nc{\abel}{\operatorname{ab}} \nc{\CH}{\operatorname{CH}}
\nc{\mcc}{\operatorname{mcc}} \nc{\Adj}{\operatorname{Adj}}
\nc{\Order}{\operatorname{O}}  \nc{\inj}{\operatorname{inj}} \nc{\proy}{\operatorname{proy}}
\nc{\vol}{\operatorname{vol}} \nc{\Diag}{\operatorname{Dg}}
\nc{\Spec}{\operatorname{Spec}} \nc{\Ima}{\operatorname{Im}} \nc{\Rea}{\operatorname{Re}}
\nc{\spann}{\operatorname{span}}

\theoremstyle{plain}
\newtheorem{theorem}{Theorem}[section]
\newtheorem{proposition}[theorem]{Proposition}
\newtheorem{corollary}[theorem]{Corollary}
\newtheorem{lemma}[theorem]{Lemma}

\theoremstyle{definition}

\theoremstyle{remark}

\newtheorem{example}[theorem]{Example}

\title{The Ricci pinching functional on solvmanifolds}

\author{Jorge Lauret} \author{Cynthia E. Will}

\address{Universidad Nacional de C\'ordoba, FaMAF and CIEM, 5000 C\'ordoba, Argentina}
\email{lauret@famaf.unc.edu.ar} \email{cwill@famaf.unc.edu.ar}

\thanks{This research was partially supported by grants from CONICET, FONCYT and Universidad Nacional de C\'ordoba}

\begin{document}

\maketitle

\begin{abstract}
We study the natural functional $F=\frac{\scalar^2}{|\Ricci|^2}$ on the space of all non-flat left-invariant metrics on all solvable Lie groups of a given dimension $n$.  As an application of properties of the beta operator, we obtain that solvsolitons are the only global maxima of $F$ restricted to the set of all left-invariant metrics on a given unimodular solvable Lie group, and beyond the unimodular case, we obtain the same result for almost-abelian Lie groups.  Many other aspects of the behavior of $F$ are clarified.
\end{abstract}

\tableofcontents

\section{Introduction}\label{intro}

Our aim is to study the behavior of the geometric functional
$$
F=\frac{\scalar^2}{|\Ricci|^2},
$$
on non-flat homogeneous Riemannian manifolds of a given dimension $n$.  Our motivation relies on the following elementary facts:

\begin{itemize}
  \item $F\leq n$ and equality holds precisely at Einstein metrics.

  \item $n-1<F(g)$ implies that $\Ricci_g$ is definite.

  \item If $M_\cca:=\sup F|_\cca $ for a particular class $\cca$ of homogeneous metrics, then any $g\in\cca$ can be at most $\sqrt{\frac{M_\cca}{n}}$-Ricci pinched.

  \item The infimum $m_\cca:=\inf F|_\cca $, on the other hand, provides the curvature estimate $|\Ricci_g|\leq m_\cca^{-\unm}\,|\scalar_g|$ for all $g\in\cca$.  The relevance of this estimate enhances when combined with the estimate $|\Riem|\leq C(n)|\Ricci|$ recently obtained in \cite{BhmLfnSmn} for any non-flat homogeneous metric, where $C(n)>0$ depends only on the dimension $n$ and $\Riem$ denotes the Riemannian curvature tensor.

  \item For the class $\cca_{G/K}$ of all $G$-invariant metrics on a homogeneous space $G/K$, the numbers $m_{\cca_{G/K}}$ and $M_{\cca_{G/K}}$ provide nice invariants of $G/K$ (up to equivariant diffeomorphism).
\end{itemize}

We focus in this paper on solvmanifolds, or more precisely, left-invariant metrics on solvable Lie groups.  Given a solvable Lie group $S$, let $\cca_S$ denote the set of all non-flat left-invariant metrics on $S$ and consider the invariants
$$
M_S:=\sup F|_{\cca_S}, \qquad m_S:=\inf F|_{\cca_S}.
$$
After some preliminary material in Section \ref{preli}, we review in Section \ref{RP-sec} what is known on the functional $F$ in the nilpotent and almost-abelian cases.  Examples of solvable Lie groups $S$ without any Einstein metric but admitting metrics as Ricci pinched as you want (i.e.\ $M_S=n$) are easily obtained by the Lie bracket variation method, as well as different kinds of examples with $m_S$ converging to zero.

Solvable Lie groups of {\it real type} (i.e.\ any adjoint map of its Lie algebra with all its eigenvalues imaginary is nilpotent) behaved specially well with respect to $F$.  As an application of results from \cite{NklNkn} and \cite{BhmLfn}, we show in Section \ref{gen} that they are geometrically characterized as follows.

\begin{theorem}\label{main0}
A solvable Lie group $S$ is of real type if and only if $m_S>0$.  In that case, $|\Riem_g|\leq C(S)|\scalar_g|$ for all $g\in\cca_S$, where $C(S)>0$ only depends on $S$.
\end{theorem}

We consider in Section \ref{unim} the unimodular case.  Recall that unimodularity is mandatory for the existence of compact quotients.  Since a unimodular solvable Lie group $S$ can never admit a Ricci negative metric (see \cite{Dtt}), one always has that $M_S\leq n-1$.  A natural question is whether $M_{\cca_{unim}}=n-1$, where $\cca_{unim}$ is the class of all unimodular solvable Lie groups of dimension $n$ endowed with a non-flat left-invariant metric.  We provide a negative answer (see Corollary \ref{cor1} below).

The absence of Einstein metrics in $\cca_{unim}$ motivates another natural problem:
\begin{quote}
What kind of metrics can be global maxima of $F|_{\cca_S}$?
\end{quote}
Strong candidates for such a distinction are Ricci solitons, which are known to be the global maxima of $F|_{\cca_S}$ in the case of a nilpotent $S$ (see \cite{einsteinsolv}).  However, in the context of solvmanifolds, one finds the stronger concept of {\it solvsoliton} (i.e.\ $\Ricci_g=cI+D$, $c\in\RR$, $D\in\Der(\sg)$), which is in some sense a Ricci soliton well adapted to the algebraic side of $S$.  The underlying solvable Lie group of any solvsoliton is necessarily of real type and $F$ attains only finitely many values at solvsolitons of a given dimension (see \cite{solvsolitons}).

Our main tool in the unimodular case will be the {\it beta operator} introduced in \cite{standard} to study Einstein solvmanifolds and further developed in \cite{solvsolitons} for solvsolitons and in \cite{alek} for general homogeneous Ricci solitons (see \cite{BhmLfn2,BhmLfn} for recent applications).  The {\it beta operator} of $(S,g)\in\cca_{unim}$ is a symmetric operator
$\beta$ of the nilradical $\ngo$ of $\sg$ such that $\beta_+:=\beta+|\beta|^2I>0$, $\tr{\beta}=-1$ and the following estimate holds:
\begin{equation}\label{est-unim-intro}
\la\Ricci_g,E_\beta\ra\geq 0, \qquad\mbox{where}\quad E_\beta:=\left[\begin{smallmatrix} 0&\\&\beta_+ \end{smallmatrix}\right] \in\glg(\sg),
\end{equation}
is written in terms of the orthogonal decomposition $\sg=\ag\oplus\ngo$.  More importantly, equality holds in \eqref{est-unim-intro} if and only if $E_\beta\in\Der(\sg)$, which is equivalent to the following structural conditions:
\begin{equation}\label{Ebeta-eq}
[\ag,\ag]=0, \qquad [\beta,\ad{\ag}|_\ngo]=0 \quad\mbox{and} \quad \beta_+\in\Der(\ngo).
\end{equation}
The beta operator is determined by the instability properties of the Lie bracket of $\ngo$ under the natural $\Gl(\ngo)$-action.  There are only finitely many possibilities for the spectrum of $\beta$, which is contained in $\QQ$ and depends only on the nilradical $\ngo$ of $\sg$, actually only on the stratum to which $\ngo$ belongs relative to the stratification of the variety of nilpotent Lie algebras defined in \cite{standard}.  We refer to Section \ref{unim} for a detailed treatment and to \cite{beta} for an updated overview of the beta operator of a general homogeneous space.

The following is our first main result.

\begin{theorem}\label{main1}
The class $\cca_{unim}$ can be partitioned in finitely many disjoint subclasses $\cca_{(m,q)}$, each one defined by a pair $(m,q)\in\NN\times\QQ$, $1\leq m\leq n$, $0\leq q<m-1$, such that the following holds:
\begin{itemize}
\item[(i)] Each class $\cca_S$ is contained in a single $\cca_{(m,q)}$.

\item[(ii)] $M_{\cca_{(m,q)}}=n-m+q<n-1$.

\item[(iii)] $M_S=n-m+q$ for any $(S,g)\in\cca_{(m,q)}$ with $S$ of real type.

\item[(iv)] $F(S,g) = n-m+q$ for some $(S,g)\in\cca_{(m,q)}$ if and only if the simply connected cover of $(S,g)$ is a solvsoliton.

\item[(v)] Solvsolitons and their quotients are the only global maxima of $F|_{\cca_{(m,q)}}$ and $F|_{\cca_S}$.
\end{itemize}
\end{theorem}

The partition of $\cca_{unim}$ is defined as follows: $(S,g)\in\cca_{(m,q)}$ if and only if $m$ is the dimension of the nilradical $\ngo$ of $\sg$ and $q=|\beta|^{-2}$, where $\beta$ in the beta operator of $(S,g)$.  If $\ngo$ is abelian, then $q=0$.  Each $\cca_{(m,q)}$ contains at least one solvsoliton.  Note that $m=n$ if and only if $S$ is nilpotent.

\begin{corollary}\label{cor1}
$M_{\cca_{unim}} = \max\{ n-m+q:\cca_{(m,q)}\ne\emptyset\} < n-1$.
\end{corollary}

\begin{corollary}\label{cor2}
Let $S$ be a unimodular solvable Lie group with nilradical $N$.  Then for any left-invariant metric $g$ on $S$,
$$
F(g) \leq \dim{S}-\dim{N}+M_N,
$$
where equality holds if and only if $g$ is a solvsoliton when pulled up to the simply connected cover of $S$.  If in addition $S$ is of real type, then $M_S=\dim{S}-\dim{N}+M_N$.
\end{corollary}

As a first step beyond the unimodular case, we study in greater depth in Section \ref{alm-abel-2} the behavior of $F$ on {\it almost-abelian} Lie groups (i.e.\ with a codimension one abelian ideal).  The set of isomorphism classes of $n$-dimensional simply connected almost-abelian Lie groups is parameterized by the space of $(n-1)\times (n-1)$-matrices up to conjugation and nonzero scaling.  The matrix $A_S$ attached to $S$ is the matrix of any nonzero adjoint map of $\sg$ restricted to the abelian ideal.  In particular, $S$ is unimodular if and only if $\tr{A_S}=0$, it is nilpotent if and only if $A_S$ is a nilpotent matrix and $S$ is non-nilpotent of real type if and only if $\Spec(A_S)$ is not completely imaginary.  The existence of special metrics on an almost-abelian $S$ can also be characterized in terms of $A_S$:
\begin{itemize}
  \item $S$ admits an Einstein metric if and only if the real semisimple part of $A_S$ is a multiple of the identity.

  \item \cite{Arr} There is a solvsoliton on $S$ if and only if $A_S$ is either semisimple or nilpotent.

  \item \cite{Jbl2} $S$ admits a Ricci soliton which is not a solvsoliton if and only if $\Spec(A_S)\subset\im\RR$ and $A_S$ is neither semisimple nor nilpotent.
\end{itemize}

We now summarize our main results in this case.

\begin{theorem}\label{main2}
Let $S$ be a simply connected almost-abelian Lie group of dimension $n$.
\begin{itemize}
  \item[(i)] If $g\in\cca_S$ is a local maximum of $F|_{\cca_S}$, then $g$ is a Ricci soliton.

  \item[(ii)] $g\in\cca_S$ is a global maximum of $F|_{\cca_S}$ if and only if $g$ is a solvsoliton.

  \item[(iii)] Suppose that $S$ admits a Ricci soliton which is not a solvsoliton.  Then $S$ admits Ricci solitons which are not local maxima of $F|_{\cca_S}$.

  \item[(iv)] Assume that $S$ is of real type and does not admit a solvsoliton.  Then $M_S=F(g_0)$, where $g_0$ is the solvsoliton on the almost-abelian Lie group $S_0$ corresponding to the semisimple part $A_{S_0}$ of $A_S$.
\end{itemize}
\end{theorem}

The proof involves the computations of the gradient of $F|_{\cca_S}$ and its second variation, as well as an application of geometric invariant theory properties of the conjugation of matrices.  Unexpected critical points of $F|_{\cca_S}$ showed up.  We note that part (ii) generalizes Theorem \ref{main1} (iv), while the behavior of Ricci solitons as potential local maxima described in parts (i) and (iii) is very intriguing.  One is certainly minded to believe that these facts will hold true in a more general context.  The existence of a Ricci soliton being a local maximum of $F|_{\cca_S}$ but not global remains open.

\vs \noindent {\it Acknowledgements.} The authors are grateful to Ramiro Lafuente for very helpful comments.

\section{Preliminaries}\label{preli}

We fix an $n$-dimensional real vector space $\sg$ endowed with an inner product $\ip$.  Let $\sca\subset\Lambda^2\sg^*\otimes\sg$ denote the algebraic subset of all Lie brackets on $\sg$ which are solvable.  Each $\mu\in\sca$ is identified with the left-invariant metric determined by $\ip$ on the simply connected solvable Lie group $S_\mu$ with Lie algebra $(\sg,\mu)$:
$$
\mu \longleftrightarrow (S_\mu,\vp).
$$
In this way, the isomorphism class $\Gl(\sg)\cdot\mu$ is identified with the set of all left-invariant metrics $S_\mu$ as follows:
$$
(S_{h\cdot\mu},\ip) \longleftrightarrow (S_\mu,\la h\cdot,h\cdot\ra), \qquad\forall h\in\Gl(\sg).
$$
Note that $h^{-1}$ is an isometric isomorphism determining an isometry between these two Riemannian manifolds.

Consider the following $\Gl(\sg)$-invariant subsets of $\sca$:
\begin{align*}
\sca_{\im\RR} :=& \left\{\mu\in\sca:\Spec(\ad_\mu{X})\subset\im\RR, \;\forall X\in\sg\right\}, \\
\sca_{\RR}  :=& \left\{\mu\in\sca:\mbox{either}\, \ad_\mu{X} \, \mbox{is nilpotent or}\, \Spec(\ad_\mu{X})\nsubseteq\im\RR, \;\forall X\in\sg\right\}, \\
\sca_{c\RR} :=& \left\{\mu\in\sca:\Spec(\ad_\mu{X})\subset\RR, \;\forall X\in\sg\right\}, \\
\sca_{unim} :=& \left\{\mu\in\sca:\tr{\ad_\mu X}=0,  \;\forall X\in\sg\right\}, \\
\nca :=& \left\{\mu\in\sca: \mu \; \mbox{is nilpotent}\right\}.
\end{align*}

The Lie algebras in $\sca_{\im\RR}$ and $\sca_\RR$ are called of {\it imaginary} and {\it real type} (or type-R), respectively (see e.g.\ \cite[Section 3]{BhmLfn}, or \cite{Jbl1}, where real type is called {\it almost completely solvable}).  The closed subset $\sca_{c\RR}$ is known in the literature as the class of {\it completely real} or {\it completely solvable} Lie algebras, and $\sca_{unim}$ is the subset of {\it unimodular} solvable Lie algebras.  It easily follows that $\sca_{\im\RR}$ is closed, $\sca_{\RR}\smallsetminus\nca$ is open in $\sca$ and $\sca_{\im\RR}\cap\sca_\RR=\nca$.  We also consider the subset
$$
\sca_{flat} := \left\{\mu\in\sca:(S_\mu,\ip)\;\mbox{is flat}\right\}.
$$
It was proved in \cite{Mln} that $\mu\in\sca_{flat}$ if and only if there is a decomposition $\sg=\ag\oplus\ngo$ such that $\mu(\ag,\ag)=0$, $\mu(\ngo,\ngo)=0$, $\mu(\ag,\ngo)\subset\ngo$ and $(\ad_\mu{X})^t=-\ad_\mu{X}$ for all $X\in\ag$.  Thus $\sca_{flat}\subset\sca_{\im\RR}$, it is only $\Or(\sg)$-invariant and $\Gl(\sg)\cdot\sca_{flat}$ consists of those $\mu\in\sca_{\im\RR}$ for which there exists a decomposition $\sg=\ag\oplus\ngo$ as above such that $\ad_\mu{X}$ is a semisimple operator for all $X\in\ag$.  Note that the Lie groups with Lie algebras in $\Gl(\sg)\cdot\sca_{flat}$ are precisely those admitting a flat metric and that $\Gl(\sg)\cdot\sca_{flat}\cap\sca_\RR=\{ 0\}$.  The following inclusions follow from the definitions,
$$
\{ 0\}\subset\sca_{flat}\subset\Gl(\sg)\cdot\sca_{flat}\subset\sca_{\im\RR}\subset\sca_{unim}, \qquad \{ 0\}\subset\nca\subset\sca_{c\RR}\subset\sca_\RR.
$$
It is worth to recall that $\mu,\lambda\in\sca_{c\RR}$ are isometric if and only if they belong to the same $\Or(\sg)$-orbit (see \cite{Alk}).  If one of them is not in the class $\sca_{c\RR}$, then they can be isometric without even being isomorphic.  On the other hand, the isometry group of any $\mu\in\sca_{c\RR}\cap\sca_{unim}$ is given by
$$
\Iso(S_\mu,\ip)=\left(\Aut(\mu)\cap\Or(\sg)\right)\ltimes S_\mu,
$$
which is equivalent to say that any isometry fixing the identity must be an automorphism of the Lie group (see \cite{GrdWls}).  Note that $\nca\subset \sca_{c\RR}\cap\sca_{unim}$.

An element $\mu\in\sca$ is called a {\it solvsoliton} when its Ricci operator satisfies that $\Ricci_\mu=cI+D$ for some $c\in\RR$ and $D\in\Der(\mu)$.  Any solvsoliton belongs to $\sca_\RR$ (see \cite[Theorem 4.8]{solvsolitons}) and any Ricci soliton metric on a solvable Lie group is isometric to a solvsoliton (see \cite[Theorem 1.1]{Jbl2}).  Moreover, any semi-algebraic soliton on a solvable Lie group (i.e.\ $\Ricci_\mu=cI+D+D^t$, $c\in\RR$, $D\in\Der(\mu)$) is a solvsoliton (see \cite[Lemma 5.3]{Jbl2}).  A nilpotent solvsoliton is called a {\it nilsoliton}.

\section{Ricci pinching functional}\label{RP-sec}

Given a nonflat homogeneous Riemannian metric $g$ on a differentiable manifold of dimension $n$, the quantity
$$
F(g):=\frac{\scalar_g^2}{|\Ricci_g|^2},
$$
where $\scalar_g$ and $\Ricci_g$ are respectively the scalar and Ricci curvature of $g$, measures in some sense how far is $g$ from being Einstein (recall that a Ricci flat homogeneous Riemannian manifold is necessarily flat; see \cite{AlkKml}).  Indeed, by Cauchy-Schwartz inequality, $F(g)\leq n$ and equality holds if and only if $g$ is Einstein.  Note that $F$ is invariant up to isometry and scaling.
It is also easy to see that if $n-1 < F(g)$, then either $\Ricci_g>0$ or $\Ricci_g<0$.  Conversely, if $\Ricci_g$ is definite, then $F(g)>1$.

Recall that the metric $g$ is called $\alpha$-{\it Ricci pinched} for $0<\alpha\leq 1$ if either $\alpha C \leq \Ricci_g \leq C$ or $-C \leq \Ricci_g \leq -\alpha C$ for some $C>0$.  Equivalently, in terms of its Ricci eigenvalues $r_1\leq\dots\leq r_n$ (counting multiplicities), $g$ is $\alpha$-Ricci pinched if and only if either $0<\Ricci_g$ and $\alpha\leq r_1/r_n$  or $\Ricci_g<0$ and $\alpha\leq r_n/r_1$.  We refer to \cite[12.3.5]{Brg} for further information on general Ricci pinching results.

It follows that $F(g)\geq \alpha^2n$ for any $\alpha$-Ricci pinched metric $g$.  In particular, any upper bound for $F$ restricted to some class $\cca$ of homogeneous metrics, say $F|_\cca\leq M_\cca$, implies that every $g\in\cca$ can be at most $\left(M_\cca/n\right)^{\unm}$-Ricci pinched.  On the other hand, a positive lower bound $m_\cca\leq F|_\cca$ gives rise to the estimate $|\Ricci_g|\leq (m_\cca)^{-\unm}\, |\scalar_g|$ for any $g\in\cca$, which implies that all the curvature on $\cca$ is actually controlled by $|\scalar|$.  Indeed, it has recently been proved in \cite[Theorem 4]{BhmLfnSmn} that there is a positive constant $C(n)$ depending only on the dimension $n$ such that $|\Riem(g)|\leq C(n)|\Ricci(g)|$, for any non-flat homogeneous metric $g$.

In this section, we aim to study the functional
$$
F:\sca\smallsetminus\sca_{flat} \longrightarrow \RR, \qquad F(\mu):=\frac{\scalar_\mu^2}{|\Ricci_\mu|^2},
$$
where $\scalar_\mu$ and $\Ricci_\mu$ are respectively the scalar curvature and Ricci operator of $\mu$ (recall that $\mu\leftrightarrow(S_\mu,\ip)$, see Section \ref{preli}).  Note that $F$ is invariant up to isometry and scaling; in particular, $F$ is $\RR^*\Or(\sg)$-invariant.

Since $\scalar_\mu< 0$ for any nonflat $\mu\in\sca$ (see \cite{Jns}),
$$
0<F(\mu)\leq n, \qquad\forall\mu\in\sca\smallsetminus\sca_{flat},
$$
and $F(\mu)=n$ if and only if $\mu$ is Einstein.  Moreover, if $n-1<F(\mu)$ then $\Ricci_\mu <0$.  Given any class $\cca\subset\sca$ of solvmanifolds, it is therefore natural to consider maximal points of $F|_\cca$ as distinguished metrics among $\cca$, as they are the closest ones to be Einstein, or the most Ricci pinched, among $\cca$.

For each $\mu\in\sca$ we define the $\Gl(\sg)$-invariant numbers,
$$
m_\mu:=\inf F(\Gl(\sg)\cdot\mu), \qquad M_\mu:=\sup F(\Gl(\sg)\cdot\mu),
$$
that is, the infimum and supremum of $F$ among all left-invariant metrics on the Lie group $S_\mu$.  These are also the infimum and supremum of $F\left(\overline{\Gl(\sg)\cdot\mu}\right)$ and $F(\Gl^+(\sg)\cdot\mu)$, and since $\Gl^+(\sg)$ is connected,
$$
(m_\mu,M_\mu)\subset F(\Gl(\sg)\cdot\mu)\subset F\left(\overline{\Gl(\sg)\cdot\mu}\right)\subset[m_\mu,M_\mu].
$$
It follows that a metric on $S_\mu$ can be at most $\left(M_\mu/n\right)^{\unm}$-Ricci pinched.  On the other hand, if $M_\mu>n-1$, then $S_\mu$ admits a Ricci negative metric, and conversely, if $\Ricci_\mu<0$ then $1<F(\mu)$.

We recall that $F$ is not defined on $\sca_{flat}$, so when we write $F(\cca)$ for some subset $\cca\subset\sca$ we always mean $F(\cca\smallsetminus\sca_{flat})$.

\begin{example}\label{level1} {\it Case} $m_\mu=M_\mu$.
It is proved in \cite{inter} that the only Lie brackets for which $\overline{\Gl(\sg)\cdot\mu}=\Gl(\sg)\cdot\mu\cup\{ 0\}$ are the orbits of
$$
\mu_{heis}(e_1,e_2)=e_3, \qquad \mu_{hyp}(e_n,e_i)=e_i, \quad i=1,\dots,n-1,
$$
and zero otherwise.  These two are the only non-abelian Lie groups admitting a unique metric up to isometry and scaling (see \cite{inter}).  Note that $S_{\mu_{hyp}}$ is isometric to the real hyperbolic space $\RR H^n$ and so $m_{\mu_{hyp}}=M_{\mu_{hyp}}=n$.  It is easy to compute that $m_{\mu_{heis}}=M_{\mu_{heis}}=\frac{1}{3}$.  Furthermore, $\mu_{heis}\in \overline{\Gl(\sg)\cdot\mu}$ for any $\mu\notin\Gl(\sg)\cdot\mu_{hyp}$, therefore
\begin{quote}
$m_\mu\leq\frac{1}{3}$ for any $\mu\in\sca$ such that $\mu\notin\Gl(\sg)\cdot\mu_{hyp}$.
\end{quote}
It is easy to see by using Lemma \ref{yuri} below that $S_{\mu_{heis}}$ and $S_{\mu_{hyp}}$ are the only non-abelian solvable Lie groups with $m_\mu=M_\mu$.
\end{example}

\begin{example}\label{dim3} {\it Thurston model geometries}.
It follows from \cite{Mln} that the behavior of $F$ on the set of all left-invariant metrics on each unimodular Lie group of dimension $3$ is given as follows:
$$
\begin{array}{c}
\SU(2):  \, 0\leq F\leq 3, \qquad
\Sl_2(\RR):  \, 0<F<2, \qquad
E(2):  \, 0<F<\frac{1}{3}, \\ \\
Sol:  \, \frac{1}{3}<F\leq 1, \qquad
Nil:  \, F\equiv \frac{1}{3}.
\end{array}
$$
\end{example}

\subsection{Nilpotent case}\label{nilp}
We refer to \cite{nilricciflow} and the references therein for detailed proofs of the results described in this example.  For any $\mu\in\nca$, $m_\mu=\frac{1}{3}$, $F\left(\overline{\Gl(\sg)\cdot\mu}\right)=[\frac{1}{3},M_\mu]$ and $\frac{1}{3}<M_\mu<n$ for any $\mu\notin\Gl(\sg)\cdot\mu_{heis}$.  Moreover, $F(\lambda)=M_\mu$ for some $\lambda\in\Gl(\sg)\cdot\mu$ if and only if $\lambda$ is a nilsoliton.  In other words, the maximum of $F$ among all left-invariant metrics on a nilpotent Lie group is attained at a nilsoliton, which is known to be unique up to isometry and scaling.  It follows that for any nonzero $\mu\in\nca$,
$$
F(\Gl(\sg)\cdot\mu)=\left\{
\begin{array}{ll}
\left(\frac{1}{3},M_\mu\right], & S_\mu\;\mbox{admits a nilsoliton and}\, \mu\notin\Gl(\sg)\cdot\mu_{heis}, \\ \\
\left(\frac{1}{3},M_\mu\right), & S_\mu\;\mbox{does not admit any nilsoliton}, \\ \\
\left\{\frac{1}{3}\right\}, & \mu\in\Gl(\sg)\cdot\mu_{heis}.
\end{array} \right.
$$
If $\mu$ has invariant $\beta$, then $M_\mu=1/|\beta|^2$ (see Section \ref{unim}).  We also have that $F(\nca)=[\frac{1}{3},M_\nca]$ for some constant $M_\nca<n-1$ depending only on $n$.  The nilsolitons with $\Ricci=cI+\Diag(1,2,\dots,n)$ have $F=\frac{n(n-1)}{2(2n+1)}$, showing for instance that $\frac{1}{5}n\leq M_\nca$ for all $7\leq n$.  In the nilpotent case, the functional $F$ is strictly increasing along any Ricci flow solution $g(t)$, unless $g(0)$ is a nilsoliton.

\subsection{Almost-abelian case}\label{alm-abel}
To study the class of {\it almost-abelian} Lie algebras (i.e.\ solvable Lie algebras with a codimension-one abelian ideal), one fixes an orthogonal decomposition $\sg=\ngo\oplus\RR e_n$ and attaches to each matrix $A\in\glg_{n-1}(\RR)$ (identified with $\glg(\ngo)$ via any fixed orthonormal basis) the Lie bracket $\mu_A$ defined by $\mu_A(\ngo,\ngo)=0$ and $\ad_{\mu_A}{e_n}|_{\ngo}=A$.  The construction covers, up to isometry, all left-invariant metrics on almost abelian Lie groups. Note that the class of almost-abelian Lie brackets is contained in $\sca_{\im\RR}\cup\sca_\RR$.  If $A$ is normal, then $A$ is isometric to its symmetric part $S(A):=\unm(A+A^t)\in\sca_\RR$ (see e.g.\ \cite[Remark 5.7]{LF}). 

It is easy to see that $\mu_B\in\Gl(\sg)\cdot\mu_A$ if and only if $B\in\RR^*\Gl_{n-1}(\RR)\cdot A$, i.e.\ $B$ is conjugate to $A$ up to scaling.  It follows that if $A=S+N$, where $S$ is semisimple, $N$ nilpotent and $[S,N]=0$, then $\mu_S,\mu_N\in\overline{\Gl(\sg)\cdot\mu_A}$.

Using the formula for the Ricci curvature of $A$ (see e.g.\ \cite[(25)]{solvsolitons} or \cite{Arr}),
\begin{equation}\label{ricA}
\Ricci_A = \left[\begin{array}{c|c}
\unm[A,A^t]-(\tr{A})S(A) & 0 \\ \hline
0 & -\tr{S(A)^2}
\end{array}\right],
\end{equation}
one obtains that
\begin{equation}\label{Faa}
F(A)=\frac{\left(\tr{S(A)^2}+(\tr{A})^2\right)^2}{\tr{S(A)^2}\left(\tr{S(A)^2}+(\tr{A})^2\right) + \unc|[A,A^t]|^2}.
\end{equation}
It follows that
$$
F(A) \leq 1+\frac{(\tr{A})^2}{\tr{S(A)^2}} \leq n,
$$
where equality holds in the first inequality if and only if $A$ is normal, and $F(A)=n$ if and only if $S(A)=aI$, $a\ne 0$, if and only if $A$ is isometric to the real hyperbolic space $\RR H^n$.  One can use that the following family of conjugate matrices satisfies
$$
F\left(\left[\begin{matrix} 1&t\\ 0&1\end{matrix}\right]\right) \underset{t\to 0}\longrightarrow 3,
$$
to easily obtain that there are many groups $S_A$ in any dimension $n$ with $M_A=n$ but without admitting an Einstein metric.

\begin{theorem}\cite[Proposition 3.3]{Arr}\label{romina}
$A$ is a solvsoliton if and only if either $A$ is normal or $A$ is nilpotent and $[A,[A,A^t]]=cA$ for some $c\in\RR$.
\end{theorem}

Concerning the behavior of $F$ close to $\sca_{flat}$, from
$$
A_t:=\left[\begin{matrix} t&-1\\ 1&-t\end{matrix}\right], \quad F(A_t)=\frac{t^4}{t^4+2t^2} \underset{t\to 0}\longrightarrow 0; \qquad
B_t:=\left[\begin{matrix} t&-1\\ 1&t\end{matrix}\right], \quad F(B_t)\equiv 3,
$$
we deduce that $F$ diverges at the flat metric $A_0=B_0$.  Since $\Spec(A_t)=\{\pm\im\sqrt{1-t^2}\}$ for any $t<1$, we have that $\mu_{A_t}\in\RR^*\Gl_2(\RR)\cdot A_0$ for any $t<1$ and hence $m_{A_0}=0$.  On the contrary, $\Spec(B_t)=\{\pm\im+t\}$, so the family $\mu_{B_t}$ is pairwise non-isomorphic.  This can be easily generalized to any dimension $n\geq 3$ (see the proof of Theorem \ref{main}); in particular, $m_A=0$ for any non-nilpotent $A$ with $\Spec(A)\subset\im\RR$ (i.e. $\mu_A\in\sca_{\im\RR}\smallsetminus\nca$).

The following family $C_t$, $0<t$, of semisimple matrices in $\sca_\RR$ given by
\begin{equation}\label{Ct}
C_t:=\left[\begin{matrix} t&-1&\\ 1&-t& \\ &&t \end{matrix}\right], \qquad F(C_t)=\frac{4t^4}{3t^4 + 2t^2} \underset{t\to 0} \longrightarrow 0,
\end{equation}
shows that $\inf\{ m_A:\mu_A\in\sca_\RR\}=0$ for any dimension $n\geq 4$.  Note that each group $S_{C_t}$, $t\ne\pm 1$, admits a solvsoliton, as any $C_t$ is conjugate to a normal matrix.  This is impossible to obtain with symmetric matrices since $m_A=\frac{1}{3}$ in that case (see Proposition \ref{ut}).

It is also easy to see that $\sup\{ M_A:A^t=A\}=n$ by just taking a family $A_t\to I$.

In the case when $\tr{A}=0$, i.e.\ $\mu_A$ unimodular, one obtains
$$
F(A)=\frac{\left(\tr{S(A)^2}\right)^2}{\left(\tr{S(A)^2}\right)^2 + \unc|[A,A^t]|^2},
$$
hence $F(A)\leq 1$ and equality holds if and only if $[A,A^t]=0$.  Thus $M_A=1$ for any $\mu_A\in\sca_\RR\smallsetminus\nca$ and it is a maximum if and only if $A$ is conjugate to a normal matrix, if and only if $A$ is semisimple.  We note that the maxima of $F$ on the set of all left-invariant metrics on $S_A$ ($\tr{A}=0$) are precisely solvsolitons, as in the nilpotent case.

\section{A geometric characterization of $\sca_\RR$}\label{gen}

The following result was obtained in \cite{NklNkn} in the proof of part (2) of Theorem 2.

\begin{lemma}\label{yuri}\cite{NklNkn}
Given $\mu\in\sca$, let $\sg=\ag\oplus\ngo$ be any decomposition such that $\ngo$ is the nilradical of $\mu$.  Then there exists $\lambda\in\overline{\Gl(\sg)\cdot\mu}$ such that $\lambda(\ag,\ag)=0$, $\lambda(\ngo,\ngo)=0$, $\lambda(\ag,\ngo)\subset\ngo$ and $\ad_\lambda{X}$ is a semisimple operator with the same spectrum as $\ad_\mu{X}$ for any $X\in\ag$.
\end{lemma}

\begin{corollary}\label{yuri-cor}
For any $\mu\in\sca\smallsetminus\sca_\RR$, there is a nonzero $\lambda\in\overline{\Gl(\sg)\cdot\mu}\cap\sca_{flat}$.
\end{corollary}

\begin{proof}
Given $\mu\in\sca\smallsetminus\sca_\RR$, we consider the Lie bracket $\lambda\in\overline{\Gl(\sg)\cdot\mu}$ given by Lemma \ref{yuri}.  Thus $\lambda\in\sca\smallsetminus\sca_\RR$ and we can take a basis $\{ Y_1,\dots,Y_r\}$ of $\ag$ such that $\Spec(\ad_\lambda{Y_1}|_{\ngo})\subset\im\RR$ and is nonzero.  Now if we take $h_k\in\Gl(\sg)$, $k\in\NN$, defined by $h_k|_{\ag}=\Diag(1,k,\dots,k)$ in terms of the basis $\{ Y_i\}$ and $h_k|_\ngo=I$, then $h_k\cdot\lambda$ converges to a Lie bracket $\overline{\lambda}$, as $k$ goes to $\infty$, whose only surviving brackets are those including the vector $Y_1$.   This implies that $0\ne\overline{\lambda}\in\overline{\Gl(\sg)\cdot\mu}\cap\sca_{flat}$, concluding the proof.
\end{proof}

The following lemma will also be very useful.

\begin{lemma}\label{ramiro}\cite[Lemma 3.4]{BhmLfn}
If $\mu\in\sca_\RR$, then $\overline{\Gl(\sg)\cdot\mu}\cap\sca_{flat}=\{ 0\}$.
\end{lemma}

The following characterization follows from Corollary \ref{yuri-cor} and Lemma \ref{ramiro}.

\begin{corollary}\label{char1}
A given $\mu\in\sca$ belongs to $\sca_\RR$ if and only if $\overline{\Gl(\sg)\cdot\mu}\cap\sca_{flat}=\{ 0\}$.
\end{corollary}

We are now ready to prove some general results on Ricci pinching of solvmanifolds.

\begin{theorem}\label{main}
\quad
\begin{itemize}
  \item[(i)] For any $\mu\in\sca_\RR$, $0<m_\mu$ and $F\left(\overline{\Gl(\sg)\cdot\mu}\right)=[m_\mu,M_\mu]$.
  \item[ ]
  \item[(ii)] For $n\geq 4$, $\inf\{ m_\mu:\mu\in\sca_\RR\}=0$; in particular, $F(\sca_\RR)=(0,n]$.
  \item[ ]
  \item[(iii)] $m_\mu=0$ for any $\mu\in\sca\smallsetminus\sca_\RR$.
\end{itemize}
\end{theorem}

\begin{proof}
Given $\mu\in\sca_{\RR}$, from Lemma \ref{ramiro} we obtain that  $\overline{\Gl^+(\sg)\cdot\mu}\cap\{\lambda:|\lambda|=1\}$ is a connected and compact subset of $\sca\smallsetminus\sca_{flat}$, so part (i) follows.  Part (ii) was proved in Example \ref{alm-abel} by using the family $C_t$ in \eqref{Ct}.

We now prove (iii).  Given $\mu\in\sca\smallsetminus\sca_\RR$, we consider the Lie bracket $\overline{\lambda}\in\overline{\Gl(\sg)\cdot\mu}$ as in the proof of Corollary \ref{yuri-cor}.  Thus there exists a basis of $\ngo$ such that $A_0:=\ad_\lambda{Y_1}|_{\ngo}$ is a block diagonal matrix with $2\times 2$ blocks of the form
$$
B(a_j)=\left[\begin{matrix} 0&-a_j\\ a_j&0\end{matrix}\right], \qquad a_j\in\RR.
$$
If $n-r$ is odd, then there is also a $1\times 1$ zero block.  For each $\epsilon>0$ we construct matrices $A_\epsilon$ with blocks
$$
B(a_j,\epsilon):=\left[\begin{matrix} \epsilon a_j&-a_j\\ a_j&-\epsilon a_j\end{matrix}\right],
$$
and consider the corresponding Lie bracket $\lambda_\epsilon$, i.e.\ $\ad_{\lambda_\epsilon}{Y_1}|_{\ngo}=A_\epsilon$ and all the other brackets zero.  Since $\Spec(B(a,\epsilon))=\{\pm\sqrt{\epsilon^2-1}a\}$ for any $a\in\RR$, we obtain that $A_\epsilon$ is conjugate to $\sqrt{1-\epsilon^2}A_0$ and so $\lambda_\epsilon\in\Gl(\sg)\cdot\overline{\lambda}$ for each $0<\epsilon<1$.  A straightforward computation gives that
$$
F(\lambda_\epsilon)=\frac{C_1\epsilon^4}{C_1\epsilon^4 + C_2\epsilon^2} \underset{\epsilon\to 0} \longrightarrow 0, \quad \mbox{where}\quad C_1:= \left(\sum a_j^2\right)^2, \quad C_2:=\sum a_j^4.
$$
Since $\lambda_\epsilon\in\overline{\Gl(\sg)\cdot\mu}$ for any $\epsilon>0$, this implies that $m_\mu=0$, concluding the proof of the theorem.
\end{proof}

The following geometric characterization follows from Theorem \ref{main}.

\begin{corollary}\label{char2}
A given $\mu\in\sca$ belongs to $\sca_\RR$ if and only if $m_\mu>0$.
\end{corollary}

Concerning Ricci versus sectional curvature, the following strong Gap Theorem has recently been proved.

\begin{theorem}\cite[Theorem 4]{BhmLfnSmn}
There is a positive constant $C(n)$ depending only on the dimension $n$ such that
$$
0<C(n)\leq\frac{|\Ricci(g)|^2}{|\Riem(g)|^2},
$$
for any non-flat homogeneous Riemannian metric $g$.
\end{theorem}

It follows from part (i) of Theorem \ref{main} that the analogous gap condition for scalar versus Ricci curvature, i.e.\
$$
0<C(S)\leq F(g),
$$
does hold for the set of all left-invariant metrics on any solvable Lie group $S$ of real type.  However, it is not true for the entire class $\sca_\RR$ by part (ii).  Recall from Example \ref{nilp} that it does hold true though for the class $\nca$ of all nilmanifolds with $C(\nca)=1/3$ for any $n$.  On the other hand, part (iii) of Theorem \ref{main} implies that the gap condition never holds for the set of left-invariant metrics on a solvable Lie group of imaginary type.

Corollary \ref{char2} can be rewritten as follows.

\begin{corollary}\label{char3}
For any non-abelian solvable Lie group $S$ of real type there exists a constant $m(S)>0$ such that
$$
|\Ricci(g)| \leq m(S)|\scalar(g)|,
$$
for any left-invariant metric $g$ on $S$.  There is no such a constant for solvable Lie groups which are not of real type.
\end{corollary}

\section{Unimodular case}\label{unim}

The discussion in this section provides a proof of Theorem \ref{main1}.  Let $S$ be an $n$-dimensional simply connected solvable Lie group with Lie algebra $\sg$ and let $\ngo$ denote the nilradical of $\sg$.  Given any left-invariant metric $\ip$ on $S$, we consider the orthogonal decomposition $\sg=\ag\oplus\ngo$.  If $\ngo$ is not abelian, then attached to $(S,\ip)$ there is a symmetric operator
$$
\beta:\ngo\longrightarrow\ngo,
$$
such that $\beta_+:=\beta+|\beta|^2I>0$, $\tr{\beta}=-1$ and the following estimate holds:
\begin{equation}\label{Ebeta}
\la E_\beta\cdot\lb,\lb\ra \geq 0, \qquad E_\beta:=\left[\begin{smallmatrix} 0&\\&\beta_+ \end{smallmatrix}\right] \in\glg(\sg).
\end{equation}
Here $E_\beta\cdot\lb := E_\beta\lb-[E_\beta\cdot,\cdot]-[\cdot,E_\beta\cdot]$, $\lb$ is the Lie bracket of $\sg$ and $\ip$ also denotes the natural inner product on $\Lambda^2\sg^*\otimes\sg$ defined by $\ip$.  Furthermore, equality holds in \eqref{Ebeta} if and only if $E_\beta\in\Der(\sg)$, if and only if
\begin{equation}\label{Ebeta-eq}
[\ag,\ag]=0, \qquad [\beta,\ad{\ag}|_\ngo]=0 \quad\mbox{and} \quad \beta_+\in\Der(\ngo).
\end{equation}
In the case that $\ngo$ is abelian, $\beta$ is not defined but the estimate \eqref{Ebeta} still holds for $\beta_+=I$ and the equality condition is equivalent to just $[\ag,\ag]=0$.  This operator $\beta$, called the {\it beta operator} of $(S,\ip)$, was introduced in \cite{standard} and is related to instability properties of the Lie bracket of $\ngo$ under the natural $\Gl(\ngo)$-action.  Estimate \eqref{Ebeta} was obtained in \cite[Lemma 4.9]{solvsolitons} and the equality condition \eqref{Ebeta-eq} follows from the algebraic properties of $\beta$ proved in \cite{standard}.  We refer to \cite{beta} for an updated overview of the beta operator of a general homogeneous space and its geometric and GIT properties.

The eigenvalues $b_1 \leq\dots\leq b_m$ of $\beta$ (counting multiplicities) are in $\QQ$ and depend only on the Lie group $S$, actually only on the nilradical $\ngo$ of $\sg$, where $m:=\dim{\ngo}$.  We say that $\ngo$ is of {\it type} $(b_1,\dots,b_m)$.  Remarkably, for any $n\in\NN$, there are only finitely many possible types among all nilpotent Lie algebras of dimension $\leq n$.  Recall that for any type $(b_1,\dots,b_m)$,
$$
b_1+\dots+b_m=-1, \qquad b_j+\sum b_i^2>0, \quad \forall j, \qquad \sum \frac{1}{b_i^2}<m-1.
$$
All possible types for $m=6$ were computed in \cite{Wll} and for $m=7$ in \cite{Frn} (see Example \ref{dim5} below for $m\leq 5$).

We assume from now on that $S$ is unimodular.  In this case, estimate \eqref{Ebeta} can be written in more geometric terms as,
\begin{equation}\label{est-unim}
\la\Ricci,E_\beta\ra\geq 0,
\end{equation}
where $\Ricci$ is the Ricci operator of $(S,\ip)$.  In what follows, we review the fundamental role that the beta operator and its spectrum play in the Ricci curvature of $(S,\ip)$, providing a proof of Theorem \ref{main1} along the way:

\begin{itemize}
\item \cite[Theorem 4.8]{solvsolitons} $(S,\ip)$ is a solvsoliton if and only if
\begin{equation}\label{soliton}
\Ricci=-\scalar_N \beta_\sg, \qquad \mbox{where}\quad
\beta_\sg:=\left[\begin{smallmatrix} -|\beta|^2I&\\&\beta \end{smallmatrix}\right],
\end{equation}
and $\scalar_N$ denotes the scalar curvature of the metric defined by $\ip$ on the simply connected nilpotent Lie group $N$ with Lie algebra $\ngo$.  This simpler way of writing the structural result \cite[Theorem 4.8]{solvsolitons} was discovered in \cite{BhmLfn2} and it easily follows from the formula for the Ricci operator given in \cite[(5)]{solvsolitons}.

\item In other words, every unimodular $(S,\ip)$ comes with a symmetric operator $\beta_\sg$ of $(\sg,\ip)$, also called the {\it beta operator} of $(S,\ip)$, which is the mandatory Ricci operator (up to scaling) in order for $(S,\ip)$ to be a solvsoliton.

\item In particular, $\left\{-\sum b_i^2, b_1,\dots,b_m\right\}$ is precisely the set of Ricci eigenvalues (up to scaling) of any potential solvsoliton on $S$, being the first one of multiplicity $n-m$ (recall that $\dim{\ngo}=m$).

\item The following estimate was also discovered in \cite{BhmLfn2}:
\begin{equation}\label{est}
|\Ricci| \geq -\scalar\left(n-m+\frac{1}{\sum b_i^2}\right)^{-\unm},
\end{equation}
where equality holds if and only if $(S,\ip)$ is a solvsoliton.  This follows by multiplying the equation $\beta_\sg=-|\beta|^2I+E_\beta$ by $\Ricci$ and using estimate \eqref{est-unim} to obtain
\begin{equation}\label{est-proof}
|\Ricci||\beta_\sg| \geq \la\Ricci,\beta_\sg\ra = -|\beta|^2\scalar + \la\Ricci,E_\beta\ra \geq -|\beta|^2\scalar.
\end{equation}
The equality condition follows from the structure result \eqref{soliton} and the equality condition \eqref{Ebeta-eq} for the estimate \eqref{Ebeta}.

\item Since $\scalar\leq 0$, estimate \eqref{est} can be rewritten as
\begin{equation}\label{bound}
F(\ip)\leq n-m+\frac{1}{\sum b_i^2}< n-1,
\end{equation}
where equality holds precisely at solvsolitons.  Note that this upper bound depends only on the spectrum $\{ b_1,\dots,b_m\}$ of $\beta$ and so it is valid for any left-invariant metric on $S$, so the supremum $M_S$ of $F$ on the set of all left-invariant metrics on $S$ is $\leq n-m+\frac{1}{\sum b_i^2}$.

\item Moreover, estimate \eqref{bound} holds for any left-invariant metric on any solvable Lie group having a nilradical of type $(b_1,\dots,b_m)$.

\item Solvsolitons are therefore global maxima of $F$ restricted to left-invariant metrics on $S$ (provided there is one).

\item It follows from \cite[Theorem 5.1]{BhmLfn} that if $S$ is of real type, then $M_S=n-m+\frac{1}{\sum b_i^2}$, and it is a maximum if and only if $S$ admits a solvsoliton.

\item In all the above items we have assumed that $\ngo$ is not abelian.  Recall that unimodularity has also been assumed in all these items.  In the case $[\ngo,\ngo]=0$, one has that $(S,\ip)$ is a solvsoliton if and only if,
$$
\Ricci=-\tfrac{\scalar}{n-m}\left[\begin{smallmatrix} -I&\\&0 \end{smallmatrix}\right],
$$
so the Ricci eigenvalues (up to scaling) of any potential non-flat solvsoliton on $S$ are $-1$ and $0$,  of multiplicity $n-m$ and $m$, respectively.  Furthermore, estimate \eqref{bound} takes the form
\begin{equation}\label{bound2}
F(\ip)\leq n-m,
\end{equation}
where equality holds if and only if $(S,\ip)$ is a solvsoliton, which implies that solvsolitons are always global maxima.  This therefore generalizes the results given at the end of Section \ref{alm-abel} in the unimodular almost-abelian case.  It also follows from \cite[Theorem 5.1]{BhmLfn} that for $S$ of real type, $M_S=n-m$, which is a maximum if and only if $S$ admits a solvsoliton.
\end{itemize}

\begin{example}\label{dim5}
A nilsoliton $(S,g)$, say with $\Ricci=cI+D$, has beta operator $D-\frac{\tr{D^2}}{\tr{D}}I$ up to scaling.  The derivation $D$ is called the {\it eigenvalue-type} of the nilsoliton or of the corresponding Einstein solvmanifold and has been computed in \cite{finding,Wll,Frn} for dimensions $m\leq 5$, $m=6$ and $m=7$, respectively.  We give in Table \ref{types} all the types for $m=5$ along with the corresponding $q$.  We denote by $\mu_1=(0,0,3 \cdot 12, 4 \cdot 13, 3 \cdot 14)$ the Lie algebra defined by
$$
\mu_1(e_1,e_2)=3e_3, \quad \mu_1(e_1,e_3)=4e_4, \quad \mu_1(e_1,e_4)=3e_5.
$$
\end{example}

\begin{table}
$$
\begin{array}{l|l|c}
\mbox{Lie algebra} & \mbox{Type} & q = 1/\sum b_i^2 \\\hline
\mu_1=(0,0,3 \cdot 12, 4 \cdot 13, 3 \cdot 14) & (-1,-\frac{1}{3},-\frac{1}{10},\frac{1}{10},\frac{1}{3}) & \frac{5}{6} \\\hline
\mu_2=(0,0,3^{\frac{1}{2}} \cdot 12, 3^{\frac{1}{2}} \cdot 13, 2^{\frac{1}{2}} \cdot 14 + 2^{\frac{1}{2}} \cdot 23) & (-\frac{4}{5},-\frac{1}{2},-\frac{1}{5},\frac{1}{10},\frac{2}{5}) & \frac{10}{11} \\\hline
\mu_3=(0,0,0,12, 2^{\frac{1}{2}} \cdot 14 + 2^{\frac{1}{2}} \cdot 23) & (-\frac{4}{5},-\frac{3}{5},-\frac{1}{5},0,\frac{3}{5}) & \frac{5}{7} \\\hline
\mu_4=(0,0,0,0, 12 + 34) & (-\frac{1}{2},-\frac{1}{2},-\frac{1}{2},-\frac{1}{2},1) & \frac{1}{2} \\\hline
\mu_5=(0,0,2 \cdot 12, 3^{\frac{1}{2}} \cdot 13, 3^{\frac{1}{2}} \cdot 23) & (-\frac{7}{10},-\frac{7}{10},-\frac{1}{5},\frac{3}{10},\frac{3}{10}) & \frac{5}{6} \\\hline
\mu_6=(0,0,3^{\frac{1}{2}} \cdot 12, 3^{\frac{1}{2}} \cdot 13, 2^{\frac{1}{2}} \cdot 14 + 2^{\frac{1}{2}} \cdot 23) & (-1,-\frac{1}{2},-\frac{1}{2},\frac{1}{2},\frac{1}{2}) & \frac{1}{2} \\\hline
\mu_7=(0,0,12,0,0) & (-1,-1,0,0,1) & \frac{1}{3} \\\hline
\mu_8=(0,0,12,13,0) & (-1,-\frac{1}{2},0,0,\frac{1}{2}) & \frac{2}{3}.
\end{array}
$$
\caption{Types for $m=5$}\label{types}
\end{table}

\section{Almost-abelian case revisited}\label{alm-abel-2}

It was shown in Section \ref{unim} that any solvsoliton is a global maximum of $F$ restricted to the set $\cca_S$ of all left-invariant metrics on a unimodular solvable Lie group $S$.  It is natural to expect that this also holds in the non-unimodular case.  However, note that even in the case of a  unimodular solvable Lie group $S$, we do not know the answer to the following questions:

\begin{itemize}
  \item Given the existence of a solvsoliton on $S$, are there other local maxima of $F|_{\cca_S}$?

  \item Assuming $S$ does not admit a solvsoliton, is there any global or local maximum of $F|_{\cca_S}$?
\end{itemize}

In this section, we study these questions on almost-abelian Lie groups (see Section \ref{alm-abel}).

\subsection{Critical points}
We start computing the critical points of the function $F:\glg_{n-1}(\RR)\smallsetminus\sog(n-1)\longrightarrow\RR$ given in \eqref{Faa}.  Consider any differentiable curve $\alpha:(-\epsilon,\epsilon)\longrightarrow\glg_{n-1}(\RR)$ such that $\alpha(0)=A$ and $\alpha'(0)=B$.  By using that
$$
\left.\ddt\right|_{t=0} S(\alpha(t))=S(B), \qquad \left.\ddt\right|_{t=0}\tr{S(\alpha(t))^2}=2\la S(A),B\ra,
$$
and
$$
\left.\ddt\right|_{t=0} [\alpha(t),\alpha(t)^t]=2S([B,A^t]), \qquad
\left.\ddt\right|_{t=0} |[\alpha(t),\alpha(t)^t]|^2 = -4\la [A,[A,A^t]],B\ra,
$$
a straightforward computation gives
\begin{equation}\label{gradFaa}
\grad(F)_A = \frac{1}{c_4(A)^2}\Big(c_1(A) I + c_2(A) S(A) + c_3(A) [A,[A,A^t]]\Big),
\end{equation}
where
$$
c_1(A)=\tr{A}\left(\tr{S(A)^2}+(\tr{A})^2\right)\Big(2\tr{S(A)^2}\left(\tr{S(A)^2}+(\tr{A})^2\right)+|[A,A^t]|^2\Big),
$$
$$
c_2(A)=\left(\tr{S(A)^2}+(\tr{A})^2\right)\Big(-2(\tr{A})^2\left(\tr{S(A)^2}+(\tr{A})^2\right) + |[A,A^t]|^2\Big),
$$
$$
c_3(A)=\left(\tr{S(A)^2}+(\tr{A})^2\right)^2>0,
$$
$$
c_4(A)=\tr{S(A)^2}\left(\tr{S(A)^2}+(\tr{A})^2\right)+\unc |[A,A^t]|^2>0.
$$

\begin{proposition}\label{critFaa2}
A matrix $A$ is a critical point of the functional $$F:\glg_{n-1}(\RR)\smallsetminus\sog(n-1)\longrightarrow\RR$$  given in \eqref{Faa} if and only if
\begin{itemize}
\item either $S(A)=cI$ for a nonzero $c\in\RR$, i.e.\ $A$ is Einstein and $F(A)=n$, a global maximum,

\item or $[A,A^t]=0$ and $\tr{A}=0$.  In this case, $A$ is a solvsoliton, $F(A)=1$ and $A$ is a global maximum of $F$ restricted to the set of unimodular matrices.
\end{itemize}
\end{proposition}

\begin{proof}
It is easy to check by using formula \eqref{gradFaa} that $\grad(F)_A=0$ for each matrix in the two items.  Conversely, if $\grad(F)_A=0$, then
$$
0=c_4(A)^2\la\grad(F)_A,A-A^t\ra=c_3(A)\la[A,[A,A^t]],A\ra=-c_3(A)|[A,A^t]|^2,
$$
and so $A$ is normal.  Now $c_2(A)\ne 0$ implies that $S(A)=cI$, and if $c_2(A)=0$, then $c_1(A)=0$ and hence $\tr{A}=0$, concluding the proof.
\end{proof}

In order to study critical points of the functional $F$ restricted to a conjugacy class $\cca(A):=\Gl_{n-1}(\RR)\cdot A$, we need to use that
$$
T_A\cca(A) = [A,\glg_{n-1}(\RR)].
$$
Recall that $\cca(A)$ contains all the metrics on the Lie group $S_A$ with Lie algebra $\mu_A$, up to isometry and scaling.

\begin{proposition}\label{critFaa}
The following conditions are equivalent:
\begin{itemize}
\item[(i)] $A$ is a critical point of $F|_{\cca(A)}$.
\item[ ]
\item[(ii)] $c_2(A)[A,A^t]=2c_3(A)[A^t,[A,[A,A^t]]]$.
\end{itemize}
\end{proposition}

\begin{proof}
$A$ is a critical point of $F|_{\cca(A)}$ if and only if $\grad(F)_A\perp[A,\glg_{n-1}(\RR)]$, that is,
$$
c_2(A)\la S(A),[A,B]\ra + c_3(A)\la[A,[A,A^t]],[A,B]\ra = 0, \qquad\forall B\in\glg_{n-1}(\RR).
$$
This is equivalent to part (ii), as was to be shown.
\end{proof}

It follows that if $A$ is a critical point of $F|_{\cca(A)}$, then
\begin{equation}\label{norm-crit}
c_2(A)|[A,A^t]|^2=2c_3(A)|[A,[A,A^t]]|^2.
\end{equation}
It is easy to check that every solvsoliton $A$ (see Theorem \ref{romina}) is a critical point of $F|_{\cca(A)}$.  The following examples show that they are not the only ones.

\begin{example}\label{RS}
Let $N$ be a nilsoliton and $C$ a nonzero matrix such that $C^t=-C$ and $[N,C]=0$.  Thus $A:=N+C$ is isometric to $N$ (see e.g.\ \cite[Remark 5.7]{LF}) and so $A$ is a Ricci soliton which is not a solvsoliton.  Nevertheless, $A$ is also a critical point of $F|_{\cca(A)}$.  Indeed, since $c_2(A)=c_2(N)$, $c_3(A)=c_3(N)$, $[A,A^t]=[N,N^t]$ and $[A^t,[A,[A,A^t]]]=[N^t,[N,[N,N^t]]]$, condition (ii) in Proposition \ref{critFaa} holds due to the fact that $N$ is a critical point of $F|_{\cca(N)}$.  Note that $F(A)=F(N)$.  It follows from \cite[Theorem 8.2]{Jbl2} that these are the only Ricci solitons among almost-abelian Lie groups which are not solvsolitons.  We note that for any $t>0$, $A_t:=tN+C\in\cca(A)$ and is also a Ricci soliton, though in pairwise different $\Or(n-1)$-orbits.
\end{example}

\begin{example}
Unexpectedly, a computational exploration using Proposition \ref{critFaa} provided several continuous families of peculiar critical points of $F|_{\cca(A)}$, including traceless $3\times 3$ matrices and the following $2\times 2$ matrices:
$$
\left[\begin{matrix} a&b\\ -b&c\end{matrix}\right], \qquad a^2+b^2+4ab<0, \quad b^2=-\frac{(a^2+b^2)^2+a^3b+ab^3}{a^2+b^2+4ab}.
$$
Note that this is a Ricci soliton if and only if either $a=c$ or $c=-a$ and $b=\pm a$.
\end{example}

It is therefore natural to focus on the study of global and local maxima of $F|_{\cca(A)}$ rather than all critical points.

\subsection{Global maxima}
We first note that when restricted to the closure of a conjugacy class, the functional $F$ has a simpler formula.  Indeed, for any fixed $A_0$, one has that $\tr{B}=\tr{A_0}$ and $\tr{B^2}=\tr{A_0^2}$ for any $B\in\overline{\cca(A_0)}$.  Recall that $\tr{S(B)^2}=\unm|B|^2+\unm\tr{B^2}$ for any matrix $B$.  Thus according to \eqref{Faa}, if $c_0:=\unm\tr{A_0^2}+(\tr{A_0})^2$ and $d_0:=\unm\tr{A_0^2}$, then
\begin{equation}\label{Faa-2}
F(B) = \frac{\left(\unm|B|^2+c_0\right)^2}{\left(\unm|B|^2+d_0\right)\left(\unm|B|^2+c_0\right) + \unc|[B,B^t]|^2}, \qquad\forall B\in\overline{\cca(A_0)}.
\end{equation}

The following results from geometric invariant theory will be very useful to study the maxima of $F|_{\cca(A)}$ (see \cite{RchSld,HnzSchStt}).  The moment map for the conjugation $\Gl_{n-1}(\RR)$-action on $\glg_{n-1}(\RR)$ is given by $m(A)=[A,A^t]/|A|^2$.  It follows that
\begin{equation}\label{mmA-1}
|A_0| \leq |B|, \quad\forall B\in\overline{\cca(A_0)} \quad \mbox{if and only if} \quad [A_0,A_0^t]=0,
\end{equation}
where equality holds if and only if $B\in\Or(n-1)\cdot A_0$ (i.e.\ $B$ is also normal).  In that case, $\cca(A_0)$ is closed; moreover, it is the unique closed $\Gl_{n-1}(\RR)$-orbit inside any $\overline{\cca(A)}$ that meets.  We also have that
\begin{equation}\label{mmA-2}
\frac{|[A_0,A_0^t]|}{|A_0|^2} \leq \frac{|[B,B^t]|}{|B|^2}, \quad\forall B\in\overline{\cca(A_0)} \quad\mbox{if and only if}\quad [[A_0,A_0^t],A_0]=cA_0, \; c\in\RR.
\end{equation}
This is equivalent to $A_0$ be a solvsoliton (see Theorem \ref{romina}).  Equality holds in \eqref{mmA-2} if and only if $B\in\Or(n-1)\cdot A_0$ (i.e.\ $B$ is also a solvsoliton).  The equivalences given in \eqref{mmA-1} and \eqref{mmA-2} are still valid for a local minimum $A_0$ (i.e.\ if we replace the condition $B\in\cca(A_0)$ by $B$ being in some neighborhood of $A_0$ in $\cca(A)$).  Moreover, solvsolitons (resp. normal matrices) are actually the only possible critical points of the functional $\frac{|[B,B^t]|}{|B|^2}$ (resp. $|B|^2$) restricted to a conjugacy class.  It is also well-known that
\begin{equation}\label{r2}
\frac{|[B,B^t]|}{|B|^2}\leq \sqrt{2}, \qquad\forall B\ne 0,
\end{equation}
where equality holds if and only if $B$ is nilpotent of rank $1$, i.e.\ $\mu_B\simeq\mu_{heis}$ (see Example \ref{level1}).  This follows from the fact that these matrices are the only ones satisfying $\cca(A)=\RR^*\Or(n-1)\cdot A$, or equivalently, $\overline{\RR^*\cca(A)}=\cca(A)\cup\{ 0\}$.

\begin{theorem}\label{solv-max}
\hspace{1cm}
\begin{itemize}
\item[(i)] Any solvsoliton $A_0$ is a global maximum of $F|_{\cca(A_0)}$.

\item[(ii)] If $A$ is either semisimple or $\tr{A}=\tr{A^2}=0$ (in particular, if $A$ is nilpotent), then $A_0\in\cca(A)$ is a local maximum of $F|_{\cca(A)}$ if and only if $A_0$ is a solvsoliton.

\item[(iii)] Suppose that $A$ is neither semisimple nor nilpotent.  Then $F|_{\cca(A)}$ does not have a global maximum.  If in addition $A$ is of real type (i.e.\ $\Spec(A)$ is not contained in $\im\RR$), then the supremum value $M_A$ of $F|_{\cca(A)}$ equals $F(A_0)$, where $A_0\in\overline{\cca(A)}$ is the unique normal solvsoliton in $\overline{\cca(A)}$ up to $\Or(n-1)$-conjugation.
\end{itemize}
\end{theorem}

\begin{proof}
If $A_0$ is normal then by \eqref{Faa-2} and \eqref{mmA-1},
$$
F(A) \leq 1+\frac{(\tr{A_0})^2}{\unm|A|^2+\unm\tr{A_0^2}} \leq 1+\frac{(\tr{A_0})^2}{\unm|A_0|^2+\unm\tr{A_0^2}} = F(A_0), \qquad\forall A\in\cca(A_0),
$$
from which part (i) for $A_0$ normal and part (ii) for $A$ semisimple follow.  Assume now that $A_0$ is a nilsoliton, it then follows from \eqref{Faa-2} and \eqref{mmA-2} that
$$
F(A) = \frac{1}{1+\frac{|[A,A^t]|^2}{|A|^4}} \leq  \frac{1}{1+\frac{|[A_0,A_0^t]|^2}{|A_0|^4}} = F(A_0), \qquad\forall A\in\cca(A_0).
$$
This also proves part (ii) in the case $\tr{A}=\tr{A^2}=0$.

It only remains to prove part (iii).  If $A=S+N$ with $S$ semisimple, $N$ nilpotent and $[S,N]=0$ then by \eqref{Faa-2} and \eqref{mmA-1}, $F(B)\leq F(A_0)$ for any $B\in\overline{\cca(A)}$, where $A_0\in\cca(S)$ is normal.  Since $A_0\in\overline{\cca(A)}$, we obtain that $M_A=F(A_0)$, and $M_A$ is not a maximum since $A_0\notin\cca(A)$, which concludes the proof.
\end{proof}

\begin{proposition}\label{ut}
If $\unm\tr{A^2+(\tr{A})^2}\geq 0$ (in particular, if $A$ is real semisimple) and $S(A)\notin\RR I$, then $m_A=\frac{1}{3}$ and it is a minimum if and only if $\mu_A\simeq\mu_{heis}$.
\end{proposition}

\begin{proof}
From \eqref{r2} and the facts that $c_0\geq d_0$ and $\unm|B|^2+d_0> 0$, we obtain that for any $B\in\cca(A_0)$,
$$
F(B) \geq  \frac{\left(\unm|B|^2+c_0\right)^2}{\left(\unm|B|^2+c_0\right)^2 + \unm|B|^4}
= \frac{1}{1 + \frac{\unm|B|^4}{\left(\unm|B|^2+c_0\right)^2}} \geq \frac{1}{3}.
$$
Thus $m_{A_0}\geq \frac{1}{3}$, but since we know that $m_{A_0}\leq \frac{1}{3}$ from Example \ref{level1}, this concludes the proof.
\end{proof}

The following behavior follows from Theorems \ref{main} and \ref{solv-max} and Proposition \ref{ut}:
$$
F(\Gl(\sg)\cdot\mu_A)=\left\{
\begin{array}{ll}
(0,M_A), & A\in\sca_{\im\RR},\; A \; \mbox{non-nilpotent}, \\ \\
(\frac{1}{3},M_A],  & A\;\mbox{nilpotent}, \\ \\
(m_A,M_A] \; \mbox{or}\; [m_A,M_A], & A\in\sca_\RR\;\mbox{semisimple}, \\ \\
(\frac{1}{3},M_A] \; \mbox{or}\; (\frac{1}{3},M_A), & \unm\tr{A^2}+(\tr{A})^2\geq 0, \\ \\
(m_A,M_A)\; \mbox{or}\; [m_A,M_A), & \mbox{otherwise},
\end{array} \right.
$$
where $0<m_A$.

\subsection{Second variation and local maxima}
In order to study the existence problem of local maxima other than solvsolitons, we need to compute the second variation of $F|_{\cca(A)}$.  We consider the curve $\alpha(t)=e^{tB}Ae^{-tB}\in\cca(A)$, which satisfies
$$
\alpha(0)=A, \quad \alpha'(0)=[B,A], \quad \alpha'(t)=[B,\alpha(t)], \quad \alpha''(0)=[B,[B,A]].
$$
It follows from \eqref{gradFaa} that
\begin{align}
\ddt F(\alpha(t)) = \frac{1}{c_4(\alpha(t))^2} &\Big(c_2(\alpha(t))\la S(\alpha(t)),\alpha'(t)\ra \label{Fprima} \\
&+ c_3(\alpha(t))\la [\alpha(t),[\alpha(t),\alpha(t)^t]],\alpha'(t)\ra\Big). \notag
\end{align}
The second derivative at a critical point is given by the following lemma.

\begin{lemma}\label{Fseg}
If $A$ is a critical point of $F|_{\cca(A)}$, then for the curve $\alpha(t)=e^{tB}Ae^{-tB}$,
\begin{align*}
c_4(A)^4\left.\frac{d^2}{dt^2}\right|_0 F(\alpha(t))
=& \unm|[A,A^t]|^2\la [A,A^t],B\ra^2 \\
&+ c_2(A)\left(\unm |[B,A]|^2+\unm\la[B^t,A],[B,A]\ra\right)  \\
&+ c_3(A)\Big(\la [[B,A],[A,A^t]],[B,A]\ra   \\
&- 2\tr{S([A^t,[B,A]])^2} + \la [A,[A,A^t]],[B,[B,A]]\ra\Big).
\end{align*}
\end{lemma}

\begin{proof}
In order to take the derivative of \eqref{Fprima}, we need the following computations:
$$
\left.\ddt\right|_0 \tr{S(\alpha(t))^2}+(\tr{\alpha(t)})^2 = \la[A,A^t],B\ra,
$$
\begin{align*}
\left.\ddt\right|_0 c_2(\alpha(t)) =& -4(\tr{A})^2\left(\tr{S(A)^2}+(\tr{A})^2\right)\la[A,A^t],B\ra + \la[A,A^t],B\ra|[A,A^t]|^2 \\
& - 4\left(\tr{S(A)^2}+(\tr{A})^2\right)\la[A,[A,A^t]],[B,A]\ra, \\
=& \left(-8(\tr{A})^2(\tr{S(A)^2}+(\tr{A})^2)+3|[A,A^t]|^2\right)\la [A,A^t],B\ra,
\end{align*}
$$
\left.\ddt\right|_0 c_3(\alpha(t)) = 2\left(\tr{S(A)^2}+(\tr{A})^2\right)\la[A,A^t],B\ra,
$$
\begin{align*}
\left.\ddt\right|_0 \la S(\alpha(t)),\alpha'(t)\ra =& \la S([B,A]),[B,A]\ra + \la S(A),\alpha''(0)\ra \\
=& \tr{S([B,A])^2} + \la S(A), [B,[B,A]]\ra \\
=& \tr{S([B,A])^2} + \unm\la[B^t,A],[B,A]\ra - \unm\tr{[B,A]^2} \\
=& \unm |[B,A]|^2 + \unm\la[B^t,A],[B,A]\ra,
\end{align*}
$$
\left.\ddt\right|_0 \la [\alpha(t),[\alpha(t),\alpha(t)^t]] = [[B,A],[A,A^t]] - 2[A,S([A^t,[B,A]])],
$$
\begin{align*}
\left.\ddt\right|_0 \la [\alpha(t),[\alpha(t),\alpha(t)^t]],\alpha'(t)\ra =& \la [[B,A],[A,A^t]],[B,A]\ra - \la 2[A,S([A^t,[B,A]])], [B,A]\ra \\ &+\la[A,[A,A^t]],[B,[B,A]]\ra\\
=& \la [[B,A],[A,A^t]],[B,A]\ra - 2\tr{S([A^t,[B,A]])^2} \\
&+\la[A,[A,A^t]],[B,[B,A]]\ra.\\
\end{align*}
All this together with Proposition \ref{critFaa} gives that
\begin{align}
c_4(A)^4\left.\frac{d^2}{dt^2}\right|_0 F(\alpha(t))
=& \unm\left(-8(\tr{A})^2(\tr{S(A)^2}+(\tr{A})^2)+3|[A,A^t]|^2\right)\la [A,A^t],B\ra^2 \notag \\
&+ c_2(A)\left(\unm |[B,A]|^2+\unm\la[B^t,A],[B,A]\ra\right) \notag \\
&- \frac{(\tr{S(A)^2}+(\tr{A})^2)c_2(A)}{c_3(A)}\la [A,A^t],B\ra^2 \notag \\
&+ c_3(A)\Big(\la [[B,A],[A,A^t]],[B,A]\ra \notag  \\
&- 2\tr{S([A^t,[B,A]])^2} + \la [A,[A,A^t]],[B,[B,A]]\ra\Big). \notag \\
=& \unm|[A,A^t]|^2\la [A,A^t],B\ra^2 + c_2(A)\left(\unm |[B,A]|^2+\unm\la[B^t,A],[B,A]\ra\right) \notag \\
&+ c_3(A)\Big(\la [[B,A],[A,A^t]],[B,A]\ra \notag  \\
&- 2\tr{S([A^t,[B,A]])^2} + \la [A,[A,A^t]],[B,[B,A]]\ra\Big), \notag
\end{align}
concluding the proof.
\end{proof}

As a first application of formula in Lemma \ref{Fseg}, we show that normal matrices are {\it non-degenerate} maxima, in the sense that the Hessian is negative definite on the orthogonal complement of $\Or(n-1)$-orbits.  Recall that $F$ is constant $\Or(n-1)$-orbits.

\begin{proposition}\label{normal-nodeg}
For any normal $A$, $\left.\frac{d^2}{dt^2}\right|_0 F(\alpha(t))\leq 0$ for any curve $\alpha(t)=e^{tB}Ae^{-tB}$ and equality holds if and only if $\alpha'(0)\in T_A\Or(n-1)\cdot A$, if and only if $[S(B),A]=0$.
\end{proposition}

\begin{proof}
If $[A,A^t]=0$, then $c_2(A)=-2(\tr{A})^2c_3(A)$ and
\begin{align*}
c_4(A)^4\left.\frac{d^2}{dt^2}\right|_0 F(\alpha(t))
=& -2c_3(A)\Big(\unm(\tr{A})^2\left(|[B,A]|^2+\la[B^t,A],[B,A]\ra\right)  \\
&+\tr{S([A^t,[B,A]])^2}\Big).
\end{align*}
Since
\begin{align*}
|[B^t,A]|^2 =& -\la B^t,[A^t,[B^t,A]]\ra = -\la B^t,[A^t,B^t],A]\ra \\
=& -\la B,[A^t,[B,A]]\ra = |[B,A]|^2,
\end{align*}
by Cauchy-Schwartz we obtain that
$$
|[B,A]|^2+\la[B^t,A],[B,A]\ra\geq 0,
$$
where equality holds if and only if $[S(B),A]=0$.  This in turn implies that $S([A^t,[B,A]])=[A^t,[S(B),A]]=0$ and that
$$
[B,A]=\unm[B-B^t,A]\in [\sog(n-1),A]=T_A\Or(n-1)\cdot A.
$$
Conversely, if $[B,A]=[C,A]$ for some skew-symmetric matrix $C$, then $[B-C,A]=0$ and so $[B-C,A^t]=0$ since $A$ is normal.  Thus $[B^t+C,A]=0$, from which follows that $[B^t,A]=-[C,A]=-[B,A]$.  This implies that $[S(B),A]=0$ and hence $\left.\frac{d^2}{dt^2}\right|_0 F(\alpha(t))=0$, concluding the proof.
\end{proof}

We now use Lemma \ref{Fseg} to compute the second variation in a preferred direction.

\begin{lemma}\label{AAt}
Let $A$ be a critical point of $F|_{\cca(A)}$ which is not normal and consider the curve $\alpha(t)=e^{tB}Ae^{-tB}$ with $B=[A,A^t]$.  Then
$\left.\frac{d^2}{dt^2}\right|_0 F(\alpha(t))\geq 0$, and equality holds if and only if $\tr{A}=0$ and $[A,[A,[A,A^t]]]=0$.
\end{lemma}

\begin{proof}
We use Lemma \ref{Fseg}, Proposition \ref{critFaa} and \eqref{norm-crit} to compute $c_4(A)^4\left.\frac{d^2}{dt^2}\right|_0 F(\alpha(t))$ as follows:
\begin{align*}
& \unm|[A,A^t]|^6 + c_2(A)|[A,[A,A^t]]|^2  \\
&+ c_3(A)\Big(-\la [[A,A^t],[A,[A,A^t]]],[A,[A,A^t]]\ra   \\
&- 2\tr{S([A^t,[A,[A,A^t]]])^2} - \la [A,[A,A^t]],[[A,A^t],[A,[A,A^t]]]\ra\Big) \\
=& \unm|[A,A^t]|^6 + c_2(A)|[A,[A,A^t]]|^2 - \frac{c_2(A)^2}{2c_3(A)}\tr{S([A,A^t])^2} \\
& - 2c_3(A)\la [[A,A^t],[A,[A,A^t]]],[A,[A,A^t]]\ra  \\
=& \unm|[A,A^t]|^6 + \frac{c_2(A)^2}{2c_3(A)}|[A,A^t]|^2 - \frac{c_2(A)^2}{2c_3(A)}|[A,A^t]|^2 \\
& - 2c_3(A)\la [[A,A^t],[A,[A,A^t]]],[A,[A,A^t]]\ra \\
=& \unm|[A,A^t]|^6 - 2c_3(A)\left(|[A^t,[A,[A,A^t]]]|^2 - |[A,[A,[A,A^t]]]|^2\right) \\
=& \unm|[A,A^t]|^6 - \frac{c_2(A)^2}{2c_3(A)}|[A,A^t]|^2 + 2c_3(A) |[A,[A,[A,A^t]]]|^2.
\end{align*}
By replacing with the formulas for the positive numbers $c_2(A)$ and $c_3(A)$, we obtain that
$$
\unm|[A,A^t]|^6 -\frac{c_2(A)^2}{2c_3(A)}|[A,A^t]|^2
$$
equals
\begin{align*}
& \unm|[A,A^t]|^6 -\unm\left(-2(\tr{A})^2(\tr{S(A)^2}+(\tr{A})^2)+|[A,A^t]|^2\right)^2|[A,A^t]|^2 \\
=&  \left(- 2(\tr{A})^4(\tr{S(A)^2}+(\tr{A})^2)^2 + 2(\tr{A})^2(\tr{S(A)^2}+(\tr{A})^2)|[A,A^t]|^2\right)|[A,A^t]|^2 \\
=&  \left(2(\tr{A})^2(\tr{S(A)^2}+(\tr{A})^2)\left(-(\tr{A})^2(\tr{S(A)^2}+(\tr{A})^2)+|[A,A^t]|^2\right)\right)|[A,A^t]|^2 \\
\geq& \left(2(\tr{A})^2(\tr{S(A)^2}+(\tr{A})^2)\left(-2(\tr{A})^2(\tr{S(A)^2}+(\tr{A})^2)+|[A,A^t]|^2\right)\right)|[A,A^t]|^2 \\
=& 2(\tr{A})^2 c_2(A) |[A,A^t]|^2 \geq 0.
\end{align*}
Thus $\left.\frac{d^2}{dt^2}\right|_0 F(\alpha(t))\geq 0$ and the equality condition is as stated, concluding the proof.
\end{proof}

\begin{corollary}
A nilsoliton $N$ is a degenerate global maximum.
\end{corollary}

\begin{proof}
Recall that $[N,[N,N^t]]=cN$, $c<0$, so it follows from Lemma \ref{AAt} that in the direction $B=[N,N^t]$, $\left.\frac{d^2}{dt^2}\right|_0 F(\alpha(t))=0$; moreover, $[B,N]$ does not belong to $T_N\Or(n-1)\cdot N$.  Indeed, if $[B,N]=[C,N]$ with $C^t=-C$, then $-cN=[C,N]$ and hence $-c|N|^2=\la[C,N],N\ra=0$, a contradiction.
\end{proof}

\begin{proposition}\label{only}
If $A$ is a local maximum of $F|_{\cca(A)}$, then $A$ is either a solvsoliton (i.e.\ global maximum) or $A$ is a Ricci soliton $N+C$ as in Example \ref{RS}
\end{proposition}

\begin{proof}
According to Lemma \ref{AAt}, any local maximum $A$ of $F|_{\cca(A)}$ other than a solvsoliton must satisfy $\tr{A}=0$ and $[A,[A,[A,A^t]]]=0$.  It then follows from Proposition \ref{critFaa} that
$$
|[A,A^t]|^2[A,A^t] = 4|S(A)|^2[S(A),Sk([A,[A,A^t]])],
$$
where $Sk(B):=\unm(B-B^t)$ denotes the skew-symmetric part of a matrix.  This implies that
$$
|[A,A^t]|^4 = 4|S(A)|^2\la Sk([A,[A,A^t]]),[S(A),[A,A^t]]\ra = 4|S(A)|^2 |Sk([A,[A,A^t]])|^2,
$$
and since by \eqref{norm-crit}, $|[A,A^t]|^4 = 2|S(A)|^2|[A,[A,A^t]]|^2$, we obtain that
$$
|[A,[A,A^t]]|^2 = \unm |Sk([A,[A,A^t]])|^2 = \unm|S([A,[A,A^t]])|^2.
$$
Thus
$$
\la S(A),S([A,[A,A^t]])\ra^2 = \unc |[A,A^t]|^4 = |S(A)|^2 |S([A,[A,A^t]])|^2,
$$
and so $S(A)=cS([A,[A,A^t]])$ for some $c>0$, that is, $\frac{1}{c}A=[A,[A,A^t]]+B$ for some skew-symmetric matrix $B$.  Note that $[A,B]=0$.  If we set $N:=A-cB$, then $[N,[N,N^t]]=\frac{1}{c}A-B=\frac{1}{c}N$, so $N$ is a nilsoliton, $cB$ is skew-symmetric, $[N,cB]=0$ and $A=N+cB$, as was to be shown.
\end{proof}

In the light of the above proposition, a natural question is if a Ricci soliton $A=N+C$ is a local maximum of $F|_{\cca(A)}$.  Recall that $N$ is a nilsoliton, $C^t=-C$ and $[N,C]=0$.  We note that $A_u:=uN+C$ are all Ricci solitons as well and that $A_u\in\cca(A_1)$ for any $u>0$.

\begin{lemma}
The Ricci soliton $A_u=uN+C$ is not a local maximum of $F|_{\cca(A_1)}$ for any sufficiently large $u>0$.
\end{lemma}

\begin{proof}
According to Lemma \ref{Fseg}, for the curve $\alpha(t)=e^{tB}A_ue^{-tB}$, we have that the number $c_4(A)^4\left.\frac{d^2}{dt^2}\right|_0 F(\alpha(t))$ equals to
\begin{align*}
&u^8\unm|[N,N^t]|^2\la [N,N^t],B\ra^2 \\
&+ u^6\unm|N|^2|[N,N^t]|^2\left(\unm |[B,uN+C]|^2+\unm\la[B^t,uN+C],[B,uN+C]\ra\right)  \\
&+ u^4\unc|N|^4\Big(u^2\la [[B,uN+C],[N,N^t]],[B,uN+C]\ra   \\
&- 2\tr{S([uN^t-C,[B,uN+C]])^2} + u^3\la [N,[N,N^t]],[B,[B,uN+C]]\ra\Big).
\end{align*}
In particular, if $B^t=B$ and $[B,N]=0$, then this becomes
$$
\unm u^4|N|^2\left(u^2|[N,N^t]|^2|[B,C]|^2 - |N|^2|[C,[B,C]]|^2\right),
$$
which is a positive number for sufficiently large $u$ provided $[B,C]\ne 0$.  This implies that $uN+C$ is not a local maximum for sufficiently large $u$.
\end{proof}

We do not know if a Ricci soliton $A_u$ can be a local maximum for some small $u$.

\subsection{Continuous extension}
We saw in Section \ref{alm-abel} that the functional
$$
F:\glg_{n-1}(\RR)\smallsetminus\sog(n-1)\longrightarrow\RR
$$
given in \eqref{Faa} can not be continuously extended to $\glg_{n-1}(\RR)$.  It is however a natural question if it is possible to continuously extend the following functionals,
$$
F:\overline{\cca(A)}\smallsetminus\sog(n-1)\longrightarrow\RR, \qquad F:\cca(A)\smallsetminus\sog(n-1)\longrightarrow\RR.
$$
Note that this only makes sense for $A$ of imaginary type and that $\overline{\cca(A)}\cap\sog(n-1)$ consists of a single $\Or(n-1)$-orbit.

Consider the following families $D_t$ and $E_t$, $0<t$, of matrices in $\sca_{\im\RR}$ given by
$$
D_t:=\left[\begin{matrix} 0&-1&&\\ 1&0&& \\ &&0&t \\ &&0&0 \end{matrix}\right], \qquad F(D_t)\equiv \frac{1}{3},
$$
and
$$
E_t:=\left[\begin{matrix} t&-\sqrt{1+t^2}&&\\ \sqrt{1+t^2}&-t&& \\ &&0&t \\ &&0&0 \end{matrix}\right],
\qquad F(E_t)=\frac{4t^4}{3t^4 + 2t^2} \underset{t\to 0} \longrightarrow 0.
$$
Since $D_t,E_t\in\cca(D_1)$ for any $t>0$ and both converge to $D_0=E_0$ as $t\to 0$, we obtain that $F|_{\overline{\cca(D_1)}\smallsetminus\sog(n-1)}$ can not be continuously extended to $\overline{\cca(D_1)}$.

For $A_0$ skew-symmetric, the continuity of the functional  $F:\cca(A_0)\longrightarrow\RR$ defined by \eqref{Faa} outside $\Or(n-1)\cdot A_0$ and $F|_{\Or(n-1)\cdot A_0}\equiv 0$ remains open.

\end{document}